\documentclass{article}
\usepackage{amsmath,amssymb,amsthm,stmaryrd,makeidx}
\usepackage[all]{xy}

\DeclareMathOperator{\der}{Der}
\DeclareMathOperator{\enm}{End}
\DeclareMathOperator{\ext}{Ext}

\DeclareMathOperator{\id}{id}

\DeclareMathOperator{\ad}{ad}
\DeclareMathOperator{\im}{im}

\DeclareMathOperator{\hmm}{Hom}

\DeclareMathOperator{\spec}{Spec}
\DeclareMathOperator{\aspec}{aSpec}

\DeclareMathOperator{\simp}{Simp}

\DeclareMathOperator{\inner}{Inner}

\newcommand{\cat}[1]{\mathbf{#1}}

\newcommand{\grmcat}[1]

\input xypic

\newtheorem{proposition}{Proposition}

\newtheorem{lemma}{Lemma}

\newtheorem{definition}{Definition}

\newtheorem{example}{Example}

\makeindex

\begin{document}
\author{Arvid Siqveland}
\title{Associative Schemes and Subschemes}

\maketitle

\begin{abstract}
In the preprint \cite{S252} we proved that there exists a localizing ring $A_M$ for $A$  an associative ring with unit, and $M=\oplus_{i=1}^rM_i$  a direct sum of $r\geq 1$ simple right $A$-modules. For a homomorphism of associative rings $A\rightarrow B$ we define the contraction of a simple $B$-module to $A.$
Then we define the set of aprime right $A$-modules $\aspec A$ to be the set of simple $A$-modules together with contractions of such. When $A$ is commutative, $\aspec A=\spec A,$ and we define a topology on $\aspec A$ such that when $A$ is commutative, this is the Zariski topology.   In the preprint \cite{S251}, we proved that when we have a topology and a localizing subcategory, there exists a sheaf of associative rings $\mathcal O_X$ on $\aspec A,$ agreeing with the usual sheaf of rings on $\spec A.$ In this text, we write out this construction, and we see that we can restrict the sheaf and topology to any subset $V\subseteq\aspec A.$ In particular, this 
proves that we can use complex varieties in real algebraic geometry, by restricting in accordance with $\mathbb R\subseteq\mathbb C.$ Thus the theory of schemes over algebraically closed fields and its associative generalization can be applied to real (algebraic) geometry.
\end{abstract}

\section{Introduction}

Real algebraic geometry can be thought of as a generalization of manifolds, where the  continuous functions are replaced by polynomials with real coefficients. Application to physics also leads to a necessary generalization to associative algebraic geometry and a generalization of continuous Riemannian metrics, see the book of O.A. Laudal, \cite{Laudal21}. One of the main problems with this, is that Riemannian metrics are defined over the reals, and the polynomial algebra over the reals, $\mathbb R[x_1,\dots,x_n],$ contains more simple modules than $\mathbb R^n.$ Because the algebraic properties governing the simple modules are better controlled by an algebra over an algebraically closed field, the main result in this text is the construction of a $\mathbb C$-algebra $A_{\mathbb R}$ such that $\simp(A_{\mathbb R})\cong\mathbb R^n.$ Thus the points in $\mathbb R^n$ is in one-to-one correspondence with the simple $A_{\mathbb R}$-modules for which $\aspec(A_{\mathbb R})$ is a fine moduli, see the book \cite{S23} or the preprint \cite{S241}. For any real manifold $M$ we define an associative variety $(\mathcal M,\mathcal O^A_L)$ over $\mathbb C$ such that the points in $M$ is in bijective correspondence with the closed points in $\mathcal M,$ and such that the charts $U$ in $\mathcal M$ corresponds to $\mathcal O^A_L(U)=A_{\mathbb R}.$

In short terms, if we want to generalize the study of manifolds by using algebraic scheme-theory, we need schemes of associative $\mathbb C$-algebras for which the set of points is in bijective correspondence with $M.$

Because of the discovery of localization  in associative rings, \cite{S252}, we include the general definition of associative schemes. Their purpose is to serve as moduli of algebraic objects which can be put in one-to-one correspondence with modules over associative schemes: That is, we need associative moduli to classify associative objects. This is most easily seen by the following.
Let $k$ be a field and consider the polynomial algebra in $n\geq 1$ variables $$A=k[x_1,\dots,x_n]=k[\underline x].$$ For a point $P=(p_1,\dots,p_n)\in k^n,$ we define the simple $A$-modules $$M(P)=A/(x_1-p_1,\dots,x_n-p_n).$$
\begin{lemma} For two points $P,Q\in k^n$ we have that $$\dim_k\ext^1_A(M_P,M_Q)=\begin{cases}n,\ P=Q\\0,\ P\neq Q.\end{cases}$$
\end{lemma}

\begin{proof} From \cite{S23} we know that $$\ext^1_A(M_1,M_2)\simeq\der_k(A,\hmm_k(M_1,M_2))/\inner$$ where $\hmm_k(M_1,M_2)$ is an $A-A$ bimodule by $a\phi(m)=\phi(am), \phi a(m)=\phi(m)a.$

When $P=Q$ we can assume $M_1=M_2=k[\underline x]/(\underline x).$ Then any inner derivation is on the form $\ad_\phi$ for which $\ad_\phi(x_i)=[\phi,x_i]=0.$ Thus the inner derivations is of dimension zero and $\ext^1_A(M_1,M_2)\simeq\der_k(A,k)$ where 
$\operatorname d_i(x_i)=1,\ i=1,\dots,n$ gives a basis for the derivations.

When $P\neq Q$ we can consider $M_{\underline 0}, M_P$ with $p_1\neq 0.$ 
For $$\delta\in\der_k(A,\hmm_k(M_{\underline 0},M_P)$$ we have $$\begin{aligned}\delta(x_ix_1)&=\delta(x_1x_i)\\
&\Updownarrow\\
\delta(x_i)x_1+x_i\delta(x_1)&=\delta(x_1)x_i+x_1\delta(x_i)\\
&\Updownarrow\\
\delta(x_i)p_1&=\delta(x_1)p_i\\
&\Updownarrow\\
\delta(x_i)&=\delta(x_i)\frac{p_i}{p_1}.
\end{aligned}$$
This is also true for the inner derivations, proving that $\dim_k\inner=1$ so that $\ext^1_A(M_P,M_Q)=0$ when $P\neq Q.$
\end{proof}

It follows from the lemma that not all finite dimensional simple modules over a noncommutative $k$-algebra can be classified by a finitely generated commutative algebra. See \cite{S23} for a lot of examples.

\section{Associative Extension and Contraction of Modules}

For an associative ring $A$ with a right $A$-module $M,$ we let the ring homomorphism $\eta^A_M:A\rightarrow\enm_\mathbb Z(M)$ denote the structure morphism, that is, for $m\in M,a\in A,\ ma=\eta^A_M(a)(m).$

\begin{definition} Let $\phi:A\rightarrow B$ be a homomorphism of associative rings. Let $M$ be a $B$-module such that the following diagram commutes:
\begin{equation}\label{extensiondiag}\xymatrix{A\ar[r]^\phi\ar[dr]_{\eta^A_M}&B\ar[d]^{\eta^B_M}\\&\enm_{\mathbb Z}(M).}\end{equation} Then the $B$ module $M$ is called an extension of the $A$-module $M$ to $B,$ and the $A$-module $M$ is called a contraction of the $B$-module $M$ to $A.$
\end{definition}

It is well known that neither the extension nor the contraction of a simple module is  necessarily simple. Looking at $\phi:\mathbb R[x]\hookrightarrow\mathbb C[x],$ the module $M=\mathbb C[x]/(x^2+1)$ is not a simple $\mathbb R[x]$-module or a simple $\mathbb C[x]$-module, even though $M_{\mathbb R}=\mathbb R[x]/(x^2+1)$ is a simple $\mathbb R[x]$-module.

When $A$ is a commutative ring and $\mathfrak p\subset A$ is a prime ideal, we have the canonical homomorphism $\iota:A\rightarrow A_{\mathfrak p}$ such that $A_\mathfrak p/\mathfrak p A_\mathfrak p$ is a simple $A_\mathfrak p$-module.
We extended this to the associative situation in \cite{S241}.

\begin{definition}\label{aprimedef} A right $A$-module $M$ is called aprime if there exists a ring homomorphism $\iota_M:A\rightarrow B$ such that $M$ is simple as $B$-module. The set of right aprime $A$-modules is denoted $\aspec A.$ 
\end{definition}

To be precise, and if somebody has not seen this before, we include the following.

\begin{lemma}\label{anotherloclemma} Let $\mathfrak m\subset A$ be a maximal ideal in a commutative ring.
Then $$A/\mathfrak m\simeq A_\mathfrak m/\mathfrak m A_\mathfrak m.$$
\end{lemma}

\begin{proof} We have a homomorphism $f:A\rightarrow A/\mathfrak m$ such for every $s\in A\setminus\mathfrak m,$ $f(s)$ is a unit. By the universal property of localization, there exists a unique homomorphism $\phi:A_\mathfrak m\rightarrow A/\mathfrak m$ such that $\phi(\frac{a}{s})=\iota(a)\iota(s)^{-1}.$ This homomorphism is clearly surjective, and its kernel is $\mathfrak m A_\mathfrak m$ giving the wanted isomorphism.
\end{proof}

\begin{lemma} When $A$ is a commutative ring there is a bijective correspondence $$\spec A\cong\aspec A.$$
\end{lemma}

\begin{proof} We send the prime ideal $\mathfrak p\subset A$ to $A_\mathfrak p/\mathfrak p A_\mathfrak p$ which is $A$-prime by definition. On the other hand, let
$M$ be an aprime $A$-module, defined by $\iota_M: A\rightarrow B$ such that $M$ is a simple $B$-module. Let $\mathfrak m\subset B$ be the maximal ideal defining $M$ as a simple $B$-module. Then $\iota_M^{-1}(\mathfrak m)$ is a prime ideal in $A.$ That these two operations are inverses to each other follows from the fact that for a maximal ideal in a ring $B$ we have $B/\mathfrak m\simeq B_{\mathfrak m}/\mathfrak m B_\mathfrak m$ as proven i Lemma \ref{anotherloclemma}.
\end{proof}

\section{Associative Schemes}

We recall the definition of associative schemes from \cite{S251} where we apply \cite{S252}. Let $A$ be an associative ring, and let $\aspec A$ be the set of aprime $A$-modules, defined in Definition \ref{aprimedef}.
For $f\in A,$ define the open set $$D(f)=\{M\in\aspec A|\ker(\eta^A_M(f))=0\},$$ and give $\aspec A$ the topology generated by the sub-basis $\{D(f)\}_{f\in A}.$ When $A$ is commutative, $\aspec A=\spec A,$ and the given topology is the Zariski topology. Let $$\eta^A_M=\oplus_{i=1}^r\eta^A_i:A\rightarrow\enm_{\mathbb Z}(\oplus_{i=1}^rM_i)$$ be a set of $r>0$ simple right $A$-modules and put $M=\oplus_{i=1}^rM_i.$ As any $A$-linear homomorphism is also $\mathbb Z$-linear, we have a canonical ring homomorphism $$\gamma:D_M=\oplus_{i=1}^r\enm_A(M_i)\rightarrow (\hmm_{\mathbb Z}(M_i,M_j))=\enm_\mathbb Z(M)=E_M.$$ This homomorphism is nonzero because it maps the identity to the identity, and so injective because each component $\enm_A(M_i)$ is a division ring. Denote by $D_M^\ast$  the set of units in $D_M,$ that is $D_M^\ast=\{s\in D_M|\gamma_i(s)\neq 0, 1\leq i\leq r\}.$

\begin{definition} The local function ring of $A$ in $M$ is the subring $A_M\subseteq E_M$ of $E_M$ generated over $\im\eta^A_M$ by the subset $\{\eta^A_M(s)^{-1}|\eta^A_M(s)\in D_M\setminus(0)\subset E_M\}.$
\end{definition}

In \cite{S252} we proved that $A_M$ fulfilled the following universal property: There exists a homomorphism $\eta:A\rightarrow A_M$ such that $\eta(s)$ is a unit whenever\newline $s\in\eta^{-1}(D_M\setminus (0)).$ If $B$ is any other associative algebra with a homomorphism $\kappa:A\rightarrow B$ such that $\kappa(s)$ is a unit whenever $s\in\eta^{-1}(D_M\setminus (0))$, and\newline $\ker\eta\subseteq\ker\kappa,$ then there exists a unique homomorphism $\rho:A_M\rightarrow B$ such that $\eta\circ\rho=\kappa.$ This implies that when $A$ is commutative and $M=A/\mathfrak m$ a simple module, then $A_M\simeq A_\mathfrak m.$ 

We now define a sheaf of associative rings on $\aspec A.$  Recall that it is proved in \cite{S250} that given a sheaf $F$ on a topological space, then $\mathcal F(U)=\underset{\underset{V\subsetneqq U}\leftarrow}\lim F(V)$ defines the sheaf $\mathcal F$ associated to $F.$

\begin{definition}\label{aschdef} Let $M=\oplus_{i=1}^rM_i$ be a direct sum of simple $A$-modules.
\begin{itemize}
\item[(i)] For each open $U\subseteq\aspec A=X$ we define the presheaf $ O_X(U)=\underset{\underset{M\subseteq U}\leftarrow}\lim\ A_M$ where $M$ is a finite subset of $U$ consisting of simple modules.
\item[(ii)] Then we define the sheaf $\mathcal O_X$ by $$\mathcal O_X(U)=\underset{\underset{V\subsetneqq U}\leftarrow}\lim O_X(V).$$
\end{itemize}
\end{definition}

Notice that (ii) in Definition \ref{aschdef} is redundant. It follows directly from the universal property of projective limits that this is a sheaf. 

\begin{proposition} Let $A$ be an associative ring. Then $$\mathcal O_{\aspec A}(\aspec A)\simeq A.$$
\end{proposition}

\begin{proof} We see that $A_f:=\mathcal O_X(D(f))\simeq A[\eta(f)^{-1}].$ The result then follows from $A\simeq A_1.$
\end{proof}

We can compose associative ring homomorphisms as in the commutative situation, defining the category of associative schemes. We notice that this definition gives an equivalence between the category of associative rings $\cat{Rings}$ and the category of affine schemes $\cat{aSch}.$

\begin{definition} An associative scheme $X$ is a topological space with a sheaf of associative rings $(X,\mathcal O_X)$ such that $X$ has a covering $X=\cup_{\alpha\in I}U_\alpha$ such that each $(U_\alpha,\mathcal O_{U_\alpha})\simeq\aspec A_\alpha$ for an associative ring $A_\alpha.$ A morphism in the category of associative schemes is a continuous map $X\rightarrow Y$ which restricts to an affine morphism on each open affine. 
\end{definition}

\section{Induced Associative Subschemes}

Let $X$ be an associative scheme, and let $\tilde Y\subseteq X$ be any subset. Give $\tilde Y$ the induced topology, and for each open affine $U=\aspec A$ let  $U_{\tilde Y}=\tilde Y\cap U$ be the intersection of $\tilde Y$ with $U.$ Now we consider the subsheaf $\mathcal O_{\tilde Y}\subseteq\mathcal O_X$ on $\tilde Y$ defined by $$\mathcal O_{\tilde Y}(U_{\tilde Y})=\underset{\underset{M\subseteq U_{\tilde Y}}\leftarrow}\lim\ A_M,$$ where $M\subseteq U$ is a finite set of simple $A$-modules.  Then the sheaf $\mathcal O_{\tilde Y}$  is a sheaf of associative rings $\mathcal O(U_{\tilde Y})$ on the induced topological space $\tilde Y.$ Let $Y$ be the topological space covered by the open affine sets $\aspec \mathcal O_{\tilde Y}(U_{\tilde Y}).$ Then $\mathcal O_{\tilde Y}$ is a sheaf on this space which by definition is an associative scheme  $(Y,\mathcal O_Y).$  Notice that $\tilde Y$ induces the topological space $Y$ with added aprime points from $X.$

\begin{definition} Let $X$ be an associative scheme and $\tilde Y\subseteq X$ any subset. Then $(Y,\mathcal O_Y)$ constructed above is called the induced associative subscheme.
\end{definition}

\begin{proposition} Let $A$ be an associative ring, let $f\in A$ and let $Y=D(f)\subseteq X=\aspec A.$ Then $$\mathcal O_Y(Y)\simeq A_f.$$ Assume that $A$ is commutative with $\mathfrak a\subseteq A$ an ideal and let  $Y=Z(\mathfrak a)\subseteq X=\aspec A.$ Then $$\mathcal O_{Z(\mathfrak a)}(Z(\mathfrak a))=A/\mathfrak a.$$ In particular, $\mathcal O_{\aspec A}(\aspec A)=A.$
\end{proposition}

\begin{proof} This follows directly from the definition.
\end{proof}

\section{Algebraic Closures of Associative Schemes}

Let $k$ be a field and consider the polynomial algebra in $n\geq 1$ variables $$A=k[x_1,\dots,x_n]=k[\underline x].$$ An algebraic set is the set of zeroes $V=Z(\mathfrak a)\subseteq k^n=\mathbb A^n_k$ in an ideal $\mathfrak a\subset A.$ An affine algebraic variety is an irreducible algebraic subset of $\mathbb A^n_k.$ The coordinate ring of $V$ is $A(V)=A/\mathfrak a.$  The maximal ideals in $A$ is known to be in bijective correspondence with $\mathbb A^n_k$ if and only if $k$ is algebraically closed, and so many results from algebraic geometry fails over the real numbers.

Call a point $x$ in an associative scheme $X$ simple if $x$ is a simple $A$-module for some open $x\in U=\spec A\subseteq X.$ Let $X\rightarrow\spec k$ be an associative scheme over a field $k,$ and let $\mathbb X=X\times_k\Bbbk$ where $\Bbbk$ is the algebraic closure of $k.$  Denote the projection to $X$ by $p:\mathbb X\rightarrow X.$

\begin{definition} Define the subset of $k$-points in $\mathbb X$ by $$\tilde{\mathbb X}(k)=\{x\in X\subseteq\mathbb X|x\text{ is simple}\}\subseteq\mathbb X.$$ Then the induced associative subscheme $\mathbb X(k)$ is called the associative subscheme of $k$-points. 
\end{definition} 

Notice in particular that $\mathbb X(k)$ is a scheme over $\Bbbk$ as we have the composition of morphisms $\mathbb X(k)\rightarrow\mathbb X\rightarrow\spec\Bbbk.$ 

\begin{lemma}\label{isolemma} Let $V_i,\ i=1,2,$ be  finite dimensional vector spaces over $k,$ and let $V_{i,\Bbbk}=V_i\otimes_k\Bbbk.$ If $\phi:V_1\rightarrow V_2$ is a $k$-linear homomorphism such that the induced $\Bbbk$-linear homomorphism $\phi\otimes\id$ is an isomorphism, then $\phi$ is already an isomorphism.
\end{lemma}

\begin{proof} Because $k\subseteq\Bbbk$ is a sub-algebra, it follows that if $\phi\otimes\id$ is an isomorphism, then $\dim_k V_1=\dim_k V_2,$ and that choosing corresponding bases, $$0\neq\det(\phi\otimes\id)=\det\phi.$$  
\end{proof}

\begin{proposition} Let $A$ be a $k$ algebra, let $\mathbb A=A\otimes_k\Bbbk.$ If $M$ is an $A$-module of finite $k$-dimension such that $M\otimes_k\Bbbk$ is a simple $\mathbb A$-module, then $M$ is simple as $A$-module.
\end{proposition}

\begin{proof} By Wedderburn's Theorem, \cite{S244}, when $\Bbbk$ is algebraically closed and $\mathbb M\otimes_k\Bbbk$ is finite dimensional, the structure homomorphism $\eta^A_{\mathbb M}\rightarrow\enm_\Bbbk(\mathbb M)$ is surjective if and only if $\mathbb M$ is simple. By Lemma \ref{isolemma} above, it follows that the structure morphism $A\rightarrow\enm_k(M)$ is surjective. This implies that $M$ is $A$-simple.
\end{proof}

In general, for an associative scheme $X$ over a field $k,$ we consider the scheme of $k$-points in its algebraic extension $\mathbb X.$ Restring to the final dimensional simple points in $\mathbb X(k),$ we get induced the associative scheme $X^{<\infty}$ parametrizing a particular class of finite dimensional points in $X$ which we can call root-points in lack of a better name.

\begin{example} The points in $\mathbb R^n$ is not in one-to-one correspondence with the maximal ideals in $\mathbb R[x_1,\dots,x_n].$ The easiest example is $(x^2+1)\subset\mathbb R[x]$ which is a maximal ideal without roots in $\mathbb R.$ However, embedding $\mathbb R^n\subset\mathbb C^n,$ the points in $\mathbb R^n$ is in bijective correspondence with the maximal ideals in $\mathbb C[x_1,\dots,x_n]$ on the form $(x_1-\alpha_1,\dots,x_n-\alpha_n)$ with $\alpha_i\in\mathbb R, 1\leq i\leq n.$ Thus $\mathbb R^n$ is in bijective correspondence with the closed points in the (commutative) scheme $\mathbb A^n_{\mathbb C}(\mathbb R)/\mathbb C.$
\end{example}

\end{document}